\documentclass[12pt]{amsart}
\usepackage[active]{srcltx}
\usepackage{verbatim}
\usepackage{epsfig,graphicx,color,mathrsfs}
\usepackage{graphicx}
\usepackage{amsmath,amssymb,amsthm,amsfonts}
\usepackage{amssymb}
\usepackage[english]{babel}

\usepackage[left=2.7cm,right=2.7cm,top=3cm,bottom=3cm]{geometry}

\newcommand{\R}{{\mathbb R}}

\newcommand{\tu}{\tilde{u}}

\renewcommand{\d }{\delta }

\renewcommand{\O}{\Omega}

\numberwithin{equation}{section}
\newtheorem{theorem}{Theorem}[section]
\newtheorem{proposition}[theorem]{Proposition}
\newtheorem{lemma}[theorem]{Lemma}

\newtheorem{remark}[theorem]{Remark}
\theoremstyle{definition}

\newcommand{\brm}{\begin{remark}\rm}
\newcommand{\erm}{\end{remark}}
\newcommand{\brms}{\begin{remark}\rm}
\newcommand{\erms}{\end{remark}}
\newcommand{\bte}{\begin{theorem}}
\newcommand{\ete}{\end{theorem}}
\newcommand{\bpr}{\begin{proposition}}
\newcommand{\epr}{\end{proposition}}
\newcommand{\ble}{\begin{lemma}}
\newcommand{\ele}{\end{lemma}}
\newcommand{\beq}{\begin{equation}}
\newcommand{\eeq}{\end{equation}}
\newcommand{\bdm}{\begin{displaymath}}
\newcommand{\edm}{\end{displaymath}}
\numberwithin{equation}{section}

\newcommand{\bos}{\begin{remark}\rm}
\newcommand{\eos}{\end{remark}}

\newcommand{\ben}{\begin{enumerate}}
\newcommand{\een}{\end{enumerate}}

\renewcommand{\d}{\delta }
\newcommand{\D }{\Delta }

\newcommand{\e }{\varepsilon }

\newcommand{\n }{\nabla }

\renewcommand{\O }{\Omega }
\newcommand{\vp }{\varphi }

\newcommand{\ov}{\overline}
\newcommand{\pa}{\partial}

\newcommand{\wt }{\tilde}

\newcommand{\be}{\begin{equation}}
\newcommand{\ee}{\end{equation}}

\title[Symmetry results]{Symmetry results for  $p$-Laplacian  systems involving a first order term}

\author[F.\ Esposito]{Francesco Esposito}

\thanks{F. Esposito and L. Montoro were partially supported by PRIN project  2017JPCAPN (Italy): {\em Qualitative and quantitative aspects of nonlinear PDEs} and also by Gruppo Nazionale per l'Analisi Matematica, la Probabilit\`a e le loro Applicazioni (GNAMPA) of the Istituto Nazionale di Alta Matematica (INdAM).
S. Merch\'an and L. Montoro were  partially supported by project MTM2016-80474-P, MINECO (Spain): {\em Problemas elipticos y parabolicos basados en potencias del Laplaciano}}

\address{Dipartimento di Matematica e Informatica
\newline\indent
Universit\`a della Calabria
\newline\indent
Ponte Pietro Bucci 31B, I-87036 Arcavacata di Rende, Cosenza, Italy}
\email{esposito@mat.unical.it}

\author[S.\ Merch\'an]{Susana Merch\'an}

\address{Departamento de Matem\'aticas 
\newline\indent
Escuela de Caminos, Canales y Puertos, Universidad Polit\'ecnica de Madrid
\newline\indent
Profesor Aranguren, 3, 28040 Madrid, Spain}
\email{susana.merchan@upm.es}

\author[L.\ Montoro]{Luigi Montoro}
\address{Dipartimento di Matematica e Informatica
\newline\indent
Universit\`a della Calabria
\newline\indent
Ponte Pietro Bucci 31B, I-87036 Arcavacata di Rende, Cosenza, Italy}
\email{montoro@mat.unical.it}


\begin{document}
\begin{abstract}
In this paper we obtain symmetry and monotonicity results for positive solutions to some $p$-Laplacian cooperative  systems in bounded domains involving  first order terms and  under zero Dirichlet boundary condition. 
%The technique that is mostly used in the proof of our result is the  moving plane method that needs to be suitably adapted to the case of quasilinear elliptic systems.
\end{abstract}

\maketitle
%\tableofcontents

%%%%%%%%%%%%%%%%%%%%%%%%%%%%%%%%%%%%%%
\medskip

\section{Introduction}\label{introdue}

The aim of this work is to get some symmetry and monotonicity results for nontrivial  solutions  $(u_1, u_2,\ldots,u_m)\in C^{1}(\ov{\O})\times C^{1}(\ov{\O})\ldots\times C^{1}(\ov{\O})$
to the following quasilinear elliptic system
\begin{equation}
\label{NVsystem}\tag{$\mathcal{S}$}
\begin{cases}
-\D_{p_i} u_i+a_i(u_i)|\nabla u_i|^{q_i}= f_i(u_1,u_2,\ldots,u_i,\ldots, u_m) & \text{in $\O$}\\
u_i>0  & \text{in $\O$}    \\
u_i=0 & \text{on $\pa\O$},  
\end{cases}
\end{equation}
where $i=1,\ldots,m$,  $p_i>1$, $q_i=\max\{1,{p_i}-1\}$,  $\Omega$ is a smooth bounded domain (connected open set) of $\R^N$, $N\geq 2$,  $\D_{p_i} u_i:={\rm div}(|\nabla u_i|^{p_i-2}\nabla u_i)$ is the $p$-Laplace operator and $a_i, f_i$  are  problem data that obey to the  set of assumptions $(hp^*)$ below.  The solution $(u_1, u_2,\ldots,u_m)$ has to be understood in the weak distributional meaning.
Our result will be obtained by means of the moving plane method, which goes back to the papers of Alexandrov \cite{A} and Serrin \cite{serrin}. In this work we use a nice variant of this technique: in particular the one of the celebrated papers of Berestycki-Nirenberg \cite{BN} and Gidas-Ni-Nirenberg \cite{GNN}, where the authors used, as essential ingredient, the maximum principle by comparing the values of the solution of the equation at two different points after a suitable reflection. Such a technique can be performed in general  convex domains providing partial monotonicity results near the boundary and symmetry properties when the domain is convex and symmetric. For simplicity of exposition and without loss of generality, since the system \eqref{NVsystem} is invariant with respect to translations and rotations,  we assume directly in all the paper that $\Omega$ is a convex domain in the $x_1$-direction and symmetric with respect to the hyperplane $\{x_1=0\}$. When $m=1$ the system \eqref{NVsystem} is reduced to   a scalar equation, that was already studied in~\cite{FMRS} in the case of $\Omega = \R^N_+$ and $1<p<2$.

\noindent The moving plane procedure was applied  to investigate symmetry properties of solutions  of cooperative semilinear elliptic systems in bounded domains, firstly by Troy \cite{troy} (see also \cite{defig1,defig2,RZ}): in this paper, the author considers the case $p_i=2$ and $a_i=0$  of~\eqref{NVsystem}.  This technique is very powerful and was adapted also in the case of cooperative semilinear systems in the half-space $\R^N_+$ by Dancer \cite{Dan} and in the entire space $\R^N$ by Busca and Sirakov~\cite{busca}. For other results regarding semilinear elliptic systems in bounded or unbounded domains, involving also critical nonlinearities, we refer to \cite{esposito}. 

The moving plane method for quasilinear elliptic equations in bounded domains was developed in several papers by Damascelli, Pacella and Sciunzi \cite{Da2,DP,DS1} and in \cite{EMS,MMPS}  for quasilinear elliptic equations involving  the   Hardy-Leray  potential  and other more general singular nonlinearities. For the case of quasilinear elliptic systems in bounded domains we refer to \cite{mrs, MSS}, where the authors considered the case $m=2$ and $a_1=a_2=0$  of  \eqref{NVsystem}.  Moreover, for other questions regarding existence, non existence and Liouville type results, in the case of (pure, i.e. $a_i=0$ in  \eqref{NVsystem}) $p$-Laplace systems,   we refer the readers to the papers (and references therein) \cite{ACM, CFMT, CMM, MP1, MP2}.

In this work we  consider the general case of $m$ $p$-Laplace equations with first order terms. 
 
To deal with the study of the qualitative properties of  solutions to \eqref{NVsystem}, first we point out  some regularity properties of the solutions to \eqref{NVsystem}, see Section \ref{prel}. Indeed the fact that solutions to $p$-Laplace equations are not in general $C^2(\Omega)$, leads to the study of the summability properties
of the second derivatives of the solutions.  Thanks to these regularity results, we are able  to prove a  weak comparison principle in small domains, i.e. Proposition \ref{pro:confr},  that is a first crucial step in the proof of the main result of the paper, namely Theorem \ref{main1} below. Moreover  we also get  some comparison and maximum principles that we will exploit in the proof of Theorem \ref{main1}. 

\

Through all the paper,  we  assume that  the following hypotheses (denoted by $({hp^*})$  in the sequel) hold:

\
\begin{itemize}
\item [$({hp^*})$]
\begin{itemize}
\item [$(i)$] For any $1\leq i\leq m$, $a_i: \mathbb{R}\to \mathbb R$ are locally Lipschitz continuous functions.
%belonging to $\mathcal{C}^1(\mathbb R)$; 
\\
\item [$(ii)$] For any $1\leq i\leq m$, $f_i: \overline{\mathbb{R}}_+^m\to \mathbb R$ are  locally $\mathcal{C}^1$ functions, i.e. $f_i\in C^1_{\rm loc}(\overline{\mathbb R}^m_+)$, and assume that  
\[f_i(t_1,t_2,...,t_m)>0,\] for all $t_i>0$. Moreover the functions $f_i$ satisfy 
\begin{equation} \label{coop}
\frac{\partial f_i}{\partial t_k}(t_1,t_2,...,t_m)\geq 0 \quad \text{for}\quad k\neq i, \,\,1\leq i,k\leq m.
\end{equation}
\end{itemize}
\end{itemize}
The  monotonicity conditions  \eqref{coop} are also  known as {\em cooperativity conditions}, see \cite{Dan, MSS, RZ, troy}.

\

\noindent Finally we have the following
\begin{theorem}
    \label{main1}
Assume that hypotheses $(hp^*)$ hold.
If $\Omega$ is convex in the  $x_1$-direction and  symmetric with
respect to the hyperplane $T_{0}=\{x \in \R^N: x_1=0 \}$,
then any solution  $(u_1, u_2,\ldots,u_m)\in C^{1}(\ov{\O})\times C^{1}(\ov{\O})\ldots\times C^{1}(\ov{\O})$  to \eqref{NVsystem} is symmetric with respect to the hyperplane $T_{0}$  and nondecreasing in the $x_1$-direction in  the  set $\Omega _0=\{x_1  <0\}$, namely
\[u_i(x_1,x_2,\cdots,x_N)=u_i(-x_1,x_2,\cdots,x_N)\quad\text{in $\Omega$}  \]
and
\begin{equation}\label{eq:akdkjsjkaslarea}
\frac {\partial  u_i}{\partial x_1}(x) \geq 0\quad\text{in $\Omega  _0$},
\end{equation}
for every $i\in \{1,\cdots,m\}.$ In particular, if $\Omega  $ is a ball, then $u_i$ are radially symmetric and radially decreasing, i.e.
\[\frac {\partial  u_i}{\partial  r}(r)<0 \quad\text{for $r\neq0$}.\] 
Moreover, if $p_i>(2N+2)/(N+2)$ for every $i\in \{1,\cdots,m\}$, then  we have 
\begin{equation}\label{eq:derivatapositiva}
\frac {\partial  u_i}{\partial x_1}(x) > 0\quad\text{in $\Omega  _0$},
\end{equation}
for every $i\in \{1,\cdots,m\}.$
\end{theorem}
\noindent The paper is organized as follows: In Section \ref{prel} we recall some preliminary results and we prove Proposition \ref{pro:confr}. The proof of the  Theorem \ref{main1} is contained  in Section \ref{sec:main}.
\section {Preliminaries} \label{prel}
\noindent In this section we are going to give  some results  for  $p$-Laplace equations involving a first order term. Through all the paper, generic fixed and numerical constants will be denoted by
$C$ (with subscript or superscript in some case) and it will be allowed to vary within a single line or formula. Moreover,  by $\mathcal{L}(\Omega)$ we will denote the Lebesgue measure of a measurable set $\Omega$.

Firstly, we recall the following inequalities (see, for example, \cite{Da2}) that we are going to use along the paper:

For all $\mu,\mu'\in\mathbb{R}^N$ with $|\mu|+|\mu'|>0$ there exist two positive constants $C,\bar{C}$ depending on $p$ such that
\begin{equation}\label{tipical}
\begin{split}
[|\mu|^{p-2}\mu-|\mu'|^{p-2}\mu'][\mu-\mu'] &\geq C(|\mu|+|\mu'|)^{p-2}|\mu-\mu'|^2,\\
||\mu|^{p-2}\mu-|\mu'|^{p-2}\mu'|&\leq\bar{C}(|\mu|+|\mu'|)^{p-2}|\mu-\mu'|.
\end{split}
\end{equation}
In the following two theorems we give some regularity results and comparison/maximum principles for the solutions to \eqref{NVsystem}.  

\begin{theorem}[See \cite{MMS,M}]\label{local1}
Let $\Omega$  a  bounded smooth  domain of $\R^N$,  $N\geq 2$, $1<p<\infty$, $q\geq \max \{p-1,1\}$ and consider $u\in C^{1}( \Omega)$ a positive weak solution to 
$$-\Delta_p u+a(u)|\nabla u|^q=f(x, u) \quad in \quad\Omega,$$
 with 
 \begin{itemize}
\item [$(i)$] $a: \mathbb{R}\to \mathbb R$ a locally Lipschitz continuous function;
\item [$(ii)$] $f\in C^1( \overline\Omega \times [0,+\infty))$.
\end{itemize}
Denoting $u_{x_i}={\partial u}/{\partial x_i}$ and {setting $\nabla u_{x_i}=0$ on $Z_u$}, for any $\Omega'\subset\Omega''\subset\subset\Omega$, we have
\begin{equation}\label{drdrd}
\begin{split}
&\int_{\Omega'} \frac{|\nabla u|^{p-2-\beta}|\nabla u_{x_i}|^2}{|x-y|^\gamma}\,dx\leqslant \mathcal{C}\qquad \forall \, i=1,\ldots,N,\\
\end{split}
\end{equation}
uniformly for any $y\in\Omega'$,
with
\[
\mathcal{C}:\,=
\mathcal{C}\Big(a, f,p, q,\beta, \gamma, \|u\|_{L^\infty(\Omega'')},\|\nabla u\|_{L^\infty(\Omega'')}\Big),
\]
for any $0\leqslant \beta <1$ and $\gamma <(N-2)$ if $N\geq3$, or $\gamma =0$ if $N=2$.\\

\noindent Moreover, if $f(x,\cdot)$ is positive in $\Omega''$, then it follows that
\begin{equation}\label{drdrdbisssetebissete}
\begin{split}
&\int_{\Omega'} \frac{1}{|\nabla u|^{r(p-1)}}\frac{1}{|x-y|^\gamma}dx\leqslant {\mathcal{C}}^*,\\
\end{split}
\end{equation}
uniformly for any $y\in\Omega'$,
with
\[
\mathcal{C^*}:\,=
\mathcal{C^*}\Big(a, f,p, q,r, \gamma, \|u\|_{L^\infty(\Omega'')},\|\nabla u\|_{L^\infty(\Omega'')}\Big),
\]
for any  $ r<1$ and $\gamma <(N-2)$ if $N\geq3$, or $\gamma =0$ if $N=2$.
\end{theorem}
In particular, these regularity results apply to the solutions $u_i$ to \eqref{NVsystem} with 
\begin{equation}\label{eq:Bwathc}f(x,u_i)=f_i(u_1,u_2,\ldots,u_i,\ldots, u_m).
\end{equation}
\begin{proof}
The proof follows exploiting  and adapting  some arguments contained in \cite{MMS, M}  to  \eqref{eq:Bwathc}-type nonlinearities. This  would imply some technicalities which we rather avoid here.
\end{proof}
\noindent For $\rho \in L^{1}(\Omega)$ and $1\leq s<\infty$,  the weighted space $H^{1,s}_\rho(\Omega)$  (with respect to $\rho$)  is
defined as the completion of $C^1(\overline \Omega)$ (or $C^{\infty }(\overline \Omega)$) with the following norm
  \begin{equation}\label{hthInorI}
\| v\|_{H^{1,s}_\rho}= \| v\|_{L^s (\Omega)}+\| \nabla v\|_{L^s (\Omega, \rho)},
\end{equation}
 where
$$
\|\nabla v\|^s_{L^s (\Omega, \rho)}:=\int_{\Omega}\rho(x)|\nabla v(x)|^s  dx.
$$ 
The space $H^{1,s}_{0,\rho}(\Omega)$ is, consequently, defined as the closure of $C^1_c(\Omega)$ (or $C^{\infty }_c(\Omega)$), with respect to the norm \eqref{hthInorI}.  We refer to \cite{DS1} for more details about weighted Sobolev spaces and also to \cite[Chapter 1]{HKM} and the references therein.
Theorem \ref{local1} provides also the right summability of the weight $|\nabla u(x)|^{p-2}$ in order to obtain a weighted Poincar\'e-Sobolev  type inequality that will be useful in the sequel. For the proof we refer to \cite[Section 3]{DS1}.
\begin{theorem}[Weighted Poincar\'e-Sobolev type inequality]\label{bvbdvvbidvldjbvlb}
Assume that hypotheses $(hp^*)$ hold and let  $(u_1, u_2,\ldots,u_m)\in C^{1}(\ov{\O})\times C^{1}(\ov{\O})\ldots\times C^{1}(\ov{\O})$ be a solution to \eqref{NVsystem}. Assume that $p_i\geq 2$  for some  $i\in \{1,\ldots,m\}$ and set $\rho_i=|\nabla u_i|^{p_i-2}$. Then, for every $w\in H^{1,2}_{0}(\Omega,\rho_i)$, we have\begin{equation}\label{Sobolev}
\|w\|_{L^2(\Omega)}\leqslant C_{P}\|\nabla w\|_{L^2(\Omega, \rho_i)}=C_{P}\left(\int_{\Omega}\rho_i\, |\nabla w|^2\right)^{\frac{1}{2}},
\end{equation}
with $C_{P}=C_{P}(\Omega)  \rightarrow 0$ if $\mathcal{L}(\Omega)\rightarrow 0$.
\end{theorem}
\noindent The following theorem collects some comparison  and maximum principles for solutions to the system \eqref{NVsystem}. We have
\begin{theorem}[See \cite{MMS,M}]\label{SCP}
Let $\Omega$  a  bounded smooth  domain of $\R^N$,  $N\geq 2$, \begin{equation}\label{eq:p>trh}
p_i>\frac{(2N+2)}{(N+2)}
\end{equation} 
and $q_i\geq \max \{p_i-1,1\}$ for  $i=1,\ldots,m$. 
Let  $(u_1, u_2,\ldots,u_m), (v_1, v_2,\ldots,v_m)\in C^{1}(\ov{\O})\times C^{1}(\ov{\O})\ldots\times C^{1}(\ov{\O})$,  with  $(u_1, u_2,\ldots,u_m)$  a solution to~\eqref{NVsystem} and   
let us assume that assumptions  $(hp^*)$ hold. 
\begin{enumerate}
\item
Then for $i=1,2,\ldots,m$, any connected domain $\Omega'\subseteq \Omega$ and  for some constant $\Lambda >0$, such that
$$
-\D_{p_i} u_i+a_i(u_i)|\nabla u_i|^{q_i}+\Lambda u_i\leq -\D_{p_i} v_i+a_i(v_i)|\nabla v_i|^{q_i}+\Lambda v_i,\quad u_i\leq v_i \quad in\quad \Omega'$$
in the weak distributional meaning,  it follows that
$$u_i<v_i\quad in\quad\Omega',$$
unless $u_i\equiv v_i$ in $\Omega'$.
\item
For any $i=1,2,\ldots,m$, for any $j=1,2,\ldots,N$,  and for  any connected domain $\Omega'\subseteq\Omega$ such that  $$\frac{\partial u_i}{\partial x_j} \geq 0 \quad in\quad \Omega',$$ it follows that 
\[\frac{\partial u_i}{\partial x_j} >0 \quad in\quad\Omega', \qquad \text{unless}\qquad  \frac{\partial u_i}{\partial x_j}= 0 \quad  \text{in} \quad\Omega'.\]
\end{enumerate}
\end{theorem}
\begin{proof} 
The part $(1)$ of the statement, follows using  the regularity results contained in Theorem \ref{local1} and  then exploiting \cite[Theorem 1.2]{MMS}.
 
 To prove the part $(2)$ we need to define the linearized equations to the system \eqref{NVsystem}. In order to do this, since $(u_1, u_2,\ldots,u_m)\in C^{1}(\ov{\O})\times C^{1}(\ov{\O})\ldots\times C^{1}(\ov{\O})$ is a weak
solution of~\eqref{NVsystem}, then we set
\begin{eqnarray*}
&&L_{(u_1,\ldots,u_m)}\big ((\partial_{x_j} u_1,\ldots,\partial_{x_j} u_i,\ldots, \partial_{x_j} u_m), (\varphi_1,\ldots, \varphi_m)\big) \\
&&=\Big(L^1_{(u_1,\ldots,u_m)}\big ((\partial_{x_j} u_1,\ldots,\partial_{x_j} u_i,\ldots, \partial_{x_j} u_m), \varphi_1\big),\ldots,\\
&& L^i_{(u_1,\ldots,u_m)}\big ((\partial_{x_j} u_1,\ldots,\partial_{x_j} u_i,\ldots, \partial_{x_j} u_m), \varphi_i\big),
\ldots, \\&&
L^m_{(u_1,\ldots,u_m)}\big ((\partial_{x_j} u_1,\ldots,\partial_{x_j} u_i,\ldots, \partial_{x_j} u_m),  \varphi_m\big)    \Big),
\end{eqnarray*}
where for  $p_i>1$,
%\begin{align*}
\begin{eqnarray*}&&L^i_{(u_1,\ldots,u_m)}\big ((\partial_{x_j} u_1,\ldots,\partial_{x_j} u_i,\ldots, \partial_{x_j} u_m), \varphi_i\big) 
\\
&&= \int_{\Omega}|\nabla u_i|^{p_i-2}(\nabla \partial_{x_j}u_i,\nabla \varphi_i)
 +(p_i-2)\int_{\Omega}|\nabla u_i|^{p_i-4}(\nabla u_i,\nabla \partial_{x_j}u_i)(\nabla u_i,\nabla
 \varphi_i)\\
 &&+\int_{\Omega}a_i'(u_i)|\nabla u_i|^{q_i}\partial_{x_j}u_i\, \varphi_i+q_i\int_{\Omega}a_i(u_i)|\nabla u_i|^{q_i-2}(\nabla u_i,\nabla \partial_{x_j}u_i)\varphi_i
 \\
 &&-\int_{\Omega}\sum_{k=1}^m\frac{\partial f_i}{\partial u_k}(u_1,\ldots,u_i,\ldots,u_m)\partial_{x_j}u_k\,\varphi_i,
 \end{eqnarray*}
for any $\varphi_1,\ldots, \varphi_m \in C^1_0 (\Omega)$.
Moreover, using the regularity results contained in Theorem~\ref{local1} (see \cite{M}), the following equation holds
\begin{equation}\label{zxcLI:VISY}
L_{(u_1,\ldots,u_m)}\big ((\partial_{x_j} u_1,\ldots,\partial_{x_j} u_i,\ldots, \partial_{x_j} u_m), (\varphi_1,\ldots, \varphi_m)\big) =0,
\end{equation}
for all $(\varphi_1,\ldots,\varphi_i,\ldots,\varphi_m)$ in $H^{1,2}_{0, \rho_{u_1}}(\Omega)\times\ldots H^{1,2}_{0, \rho_{u_i}}(\Omega)\times \ldots H^{1,2}_{0, \rho_{u_m}}(\Omega)$
where
$$
\rho_{u_i}(x):= |\n u_i(x)|^{p_i-2},\qquad i=1,\ldots,m.
$$
Since $f_i$  are locally $\mathcal{C}^1$ functions and $\|u_i\|_{L^{\infty}(\Omega)}\leq C$ for any $i\in \{1,\ldots, m\}$, there exists a
positive constant $\Theta$ such that
\begin{equation}\label{eq:find1}
\frac{\partial f_i}{\partial u_i}+\Theta\geq 0\,\,\text{for all $u_1,u_2,\dots,u_m>0$}.
\end{equation}
Moreover, in light of \eqref{coop}
we have  
\begin{equation}\label{eq:find2}
\frac{\partial f_i}{\partial u_k}(u_1,\ldots,u_i,\ldots,u_m)\geq 0 \end{equation}
for $i\neq k$.
Therefore, using \eqref{eq:find1} and \eqref{eq:find2} and taking into account \eqref{zxcLI:VISY}, it follows, for all $j=1,\ldots,N$ and for all $i=1,\ldots,m$, that $\partial_{x_j} u_i$ 
are nonnegative functions solving the  inequalities
\begin{eqnarray*} &&\int_{\Omega}|\nabla u_i|^{p_i-2}(\nabla \partial_{x_j}u_i,\nabla \varphi_i)
 +(p_i-2)\int_{\Omega}|\nabla u_i|^{p_i-4}(\nabla u_i,\nabla \partial_{x_j}u_i)(\nabla u_i,\nabla
 \varphi_i)\\
 &&+\int_{\Omega}a_i'(u_i)|\nabla u_i|^{q_i}\partial_{x_j}u_i\, \varphi_i+q_i\int_{\Omega}a_i(u_i)|\nabla u_i|^{q_i-2}(\nabla u_i,\nabla \partial_{x_j}u_i)\varphi_i
 \\
 &&+\Theta\int_{\Omega} \partial_{x_j}u_i \, \varphi_i\geq 0
 \end{eqnarray*}
 for all nonnegative test functions $\varphi_i\geq 0$.
 
Therefore,  we can apply~\cite[Theorem 3.1]{M} to each
$\partial_{x_j} u_i$  separately obtaining that,
for every $s>1$ sufficiently close to $1$ and some positive $\delta$ sufficiently small,
there exists a positive constant $C$ such that
\begin{equation}
    \label{CLAIMIIwINTY}
\|\partial_{x_j} u_i\|_{L^s (B(x, 2\delta))}\leq C_1\inf_{B(x, \delta)}\partial_{x_j} u_i.
\end{equation}
Then the sets $\{x\in\Omega': \partial_{x_j} u_i=0\}$  are
both closed (by continuity) and open (via inequalitity~\eqref{CLAIMIIwINTY})
in the domain $\Omega'$. This yields the assertion.
\end{proof}
\begin{remark}
\label{rem:SCP}
We point out that Theorem  \ref{SCP} holds without any a priori assumption on the critical set of the solution $(u_1, u_2,\ldots,u_m)$, that is, the set where the gradients $\nabla u_i$ vanish. On the other hand, though, condition \eqref{eq:p>trh} can be removed when we work in connected domain $\Omega'$ such that $\nabla u_i\neq 0$ for all $x\in \Omega'$  and for all $i\in \{1,\ldots, m\}$. Indeed, the statements $(1)$ and $(2)$ of Theorem \ref{SCP} hold in the whole range $p_i>1$. 
\end{remark}
\noindent Note that the positivity of $f(x,\cdot)$, is actually needed to obtain  \eqref{drdrdbisssetebissete}. Furthermore,  by~\eqref{drdrdbisssetebissete} it follows that
the critical set $\{x\in \Omega \, : \,\nabla u(x)=0\}$ has zero Lebesgue measure.\\

\noindent An essential tool in the proof of Theorem \ref{main1} is the  Proposition \ref{pro:confr} below, i.e. a weak comparison principle in small domains. To prove it, we start giving the following assumptions:

\

\begin{itemize}
\item [$(\ast)$] We suppose that  $(u_1, u_2,\ldots,u_m) \in C^{1}(\overline{\Omega}_1)\times C^{1}(\overline{\Omega}_1)\ldots\times C^{1}(\overline{\Omega}_1)$ is a solution to \eqref{NVsystem} in the smooth bounded domain $\Omega_1\subset \mathbb R^N$ and $(\tilde u_1, \tilde u_2,\ldots,\tilde u_m) \in C^{1}(\overline{\Omega}_2)\times C^{1}(\overline{\Omega}_2)\ldots\times C^{1}(\overline{\Omega}_2)$ is a solution to \eqref{NVsystem} in the smooth bounded domain $\Omega_2 \subset \mathbb R^N$, with $$\Omega_1\cap\Omega_2\neq \emptyset.$$ 
\end{itemize}

\

\begin{proposition}
    \label{pro:confr}
Assume that $(\ast)$ holds,  $p_i>1$, $q_i=\max\{1,{p_i}-1\}$ for every $i\in \{1,2,\ldots, m\}$ and  let  $\Omega\subset\Omega_1\cap\Omega_2$ be  a  connected set.   Then, there exists a positive number $\d$, depending upon $m, p_i,q_i,a_i, f_i,\|u_i\|_{L^\infty(\Omega)},\|\nabla u_i\|_{L^{\infty}(\Omega)}, \|\nabla \wt{u}_i\|_{L^{\infty}(\Omega)}$,  $i=1,2, \ldots, m$, such that if $\O_0\subset\Omega$ with
$$
\mathcal{L}(\O_0) \leq \d\quad \text{and} \quad u_i\leq \tilde u_i\,\,\text{on $\pa\O_0$ for every $i\in \{1,\ldots,m\}$,}
$$
 then
$$
u_i\leq\tilde u_i\,\,\text{in $\O_0$,}
$$
for every $i\in \{1,\ldots,m\}$.
\end{proposition}

\begin{proof}
Let us set
$$
U_i=(u_i- \tu_i)^+.
$$
We will prove the result by showing that 
$$(u_i-\tu_i)^+\equiv 0,$$
for every $i\in \{1,2, \ldots, m\}$. Since $u_i \leq\tilde{u_i}$ on $\pa\O_0$, then
the functions $(u_i-\tu_i)^+$ belong to $W^{1,p_i}_0(\O_0)$. Therefore,  since $u_i, \tilde u_i$ are both  weak solutions to \eqref{NVsystem} in $\Omega$,  for all $\vp \in C^\infty_c(\O)$  we have
\begin{align}
    \label{eq:wu}
& \int_\O|\n u_i|^{p_i-2}(\n u_i,\n \vp) dx+\int_\O a_i(u_i)|\n u_i|^{q_i}\vp dx= \int_\O f_i(u_1,u_2,\ldots,u_m) \vp dx\\\nonumber
&\text{and}
    \\\label{eq:wtu}
& \int_\O|\n \wt{u}_i|^{p_i-2}(\n \wt{u}_i,\n \vp) dx+\int_\O a_i(\wt{u}_i)|\n \wt{u}_i|^{q_i}\vp dx= \int_\O
f_i(\wt{u}_1,\wt{u}_2,\ldots,\wt{u}_m) \vp dx,
\end{align}
for $i=1,2, \ldots, m$. 
By a density argument, we can put respectively $\vp=(u_i-\tu_i)^+$ in
equations~\eqref{eq:wu} and~\eqref{eq:wtu}. Subtracting, we get for any $i$
\begin{align}
    \label{eq:sub1}
& \int_{\O_0}\big( |\n u_i|^{p_i-2}\n u_i -|\n\wt{u}_i|^{p_i-2}\n \wt{u}_i,
\n (u_i-\tu_i)^+ \big)	\, dx\\ \nonumber
&+\int_{\O_0}\big(a_i(u_i)|\n u_i|^{q_i}-a_i(\wt{u}_i)|\n \wt{u}_i|^{q_i} \big) (u_i-\tu_i)^+ \, dx \\ \nonumber
& =\int_{\O_0}[f_i(u_1,u_2,\ldots,u_m)-f_i(\wt{u}_1,\wt{u}_2,\ldots,\wt{u}_m)](u_i-\tu_i)^+ \, dx.
\end{align}

\noindent The second term on the left hand side of \eqref{eq:sub1} can be estimated as follows
\begin{equation*}
\begin{split}
&\left\lvert \int_{\O_0} \big(a_i(u_i)|\n  u_i|^{q_i}-a_i(\wt{u}_i)|\n \wt{u}_i|^{q_i}\big)(u_i-\tu_i)^+ \, dx \right\lvert \\
&=\left\lvert\int_{\O_0}\big(a_i(u_i)|\n u_i|^{q_i}-a_i(u_i)|\n \wt{u}_i|^{q_i}+a_i(u_i)|\n \wt{u}_i|^{q_i}-a_i(\wt{u}_i)|\n \wt{u}_i|^{q_i}\big)(u_i-\tu_i)^+ \, dx\right\lvert\\
&\leq\int_{\O_0}|a_i(u_i)|\big||\n u_i|^{q_i}-|\n \wt{u}_i|^{q_i}\big| (u_i-\tu_i)^+ \, dx +\int_{\O_0}|\n\wt{u}_i|^{q_i}(a_i(u_i)-a_i(\wt{u}_i)) (u_i-\tu_i)^+ \, dx.
\end{split}
\end{equation*}
\noindent
Since $a_i$ is a locally Lipschitz continuous function (see $({hp^*})$),  it follows that    there exists a  positive constant $K_{a_i}=K_{a_i}(\|u_i\|_{L^{\infty}(\Omega)})$ such that for every $ u_i \in [0,\|u_i\|_{L^{\infty}(\Omega)}]$
\begin{equation*}
 \vert a_i(u_i) \vert \le K_{a_i}.\end{equation*}
Moreover denoting by 
$L_{a_i}=L_{a_i}(\|u_i\|_{L^{\infty}(\Omega)})$ the Lipschitz constant of $a_i$, 
 we obtain
\begin{equation}\label{eq:rayplay1}
\begin{split}
&\left\lvert \int_{\O_0}\big(a_i(u_i)|\n u_i|^{q_i}-a_i(\wt{u}_i)|\n \wt{u}_i|^{q_i}\big) (u_i-\tu_i)^+ \, dx \right\lvert\\ &\leq K_{a_i}\int_{\O_0}\big| |\n u_i|^{q_i}-|\n \wt{u}_i|^{q_i}\big| (u_i-\tu_i)^+ \, dx \\
&+C(q_i, L_{a_i}, \|\n \tilde u_i\|_{L^{\infty}(\Omega)})\int_{\O_0}[(u_i-\tu_i)^+]^2 \,dx.
\end{split}
\end{equation}
By the mean value's theorem and taking into account that $q_i\geq 1$, it follows that 
\begin{equation}\nonumber
\begin{split}
&K_{a_i}\int_{\O_0}\big| |\n u_i|^{q_i}-|\n \wt{u}_i|^{q_i}\big| (u_i-\tu_i)^+ \, dx \\
&\leq C(q_i,K_{a_i})\int_{\O_0}(|\n u_i|+|\n \wt{u}_i|)^{q_{i}-1}|\nabla(u_i-\wt{u}_i)^+| (u_i-\tu_i)^+ \, dx.
\end{split}
\end{equation}
\noindent The last term (recall that $q_i\geq \max\{1,p_i-1\}$) can be written as follows,
\begin{eqnarray}\label{eq:bshahhsowehcaa}
&&C\int_{\O_0}(|\n u_i|+|\n \wt{u}_i|)^{q_{i}-1}|\nabla(u_i-\wt{u}_i)^+| (u_i-\tu_i)^+\, dx\\\nonumber
&&= C\int_{\O_0}\frac{(|\nabla u_i|+|\nabla \wt{u}_i|)^{q_i-1}}{(|\nabla u_i|+|\nabla \wt{u}_i|)^\frac{p_i-2}{2}}(|\nabla u_i|+|\nabla \wt{u}_i|)^\frac{p_i-2}{2}|\nabla (u_i-\tu_i)^+| (u_i-\tu_i)^+\,dx\\\nonumber
&&\leq C\int_{\O_0}(|\nabla u_i|+|\nabla \wt{u}_i|)^\frac{p_i-2}{2}|\nabla (u_i-\tu_i)^+|(u_i-\tu_i)^+ \,dx,
\end{eqnarray}
with $C=C(p_i,q_i,K_{a_i},\|\nabla u_i\|_{L^{\infty}(\Omega)}, \|\nabla \wt{u}_i\|_{L^{\infty}(\Omega)})$ is a positive constant. Exploiting Young's inequality in the right hand side of \eqref{eq:bshahhsowehcaa} we finally obtain 
\begin{eqnarray}\nonumber
&&C\int_{\O_0}(|\n u_i|+|\n \wt{u}_i|)^{q_{i}-1}|\nabla (u_i-\tu_i)^+||(u_i-\tu_i)^+|\,dx\\\nonumber
&&\leqslant \varepsilon C\int_{\O_0}(|\nabla u_i|+|\nabla \wt{u}_i|)^{p_i-2}|\nabla (u_i-\tu_i)^+|^2\,dx  \\\nonumber
&&+\frac{C}{\varepsilon} \int_{\O_0}[(u_i-\tu_i)^+]^2\,dx.
\end{eqnarray}
Therefore, collecting the previous estimates, from \eqref{eq:rayplay1}, we obtain
\begin{equation} \label{eq:rayplay2}
\begin{split}
&\left\lvert \int_{\O_0} \big(a_i(u_i)|\n  u_i|^{q_i}-a_i(\wt{u}_i)|\n \wt{u}_i|^{q_i}\big)(u_i-\tu_i)^+ \, dx \right\lvert \\
&\leqslant \varepsilon C\int_{\O_0}(|\nabla u_i|+|\nabla \wt{u}_i|)^{p_i-2}|\nabla (u_i-\tu_i)^+|^2\,dx  \\\nonumber
&+\frac{C}{\varepsilon} \int_{\O_0}[(u_i-\tu_i)^+]^2\,dx.
\end{split}
\end{equation}
\noindent Finally, using \eqref{tipical} and fixing  $\varepsilon$ sufficiently small, from \eqref{eq:sub1} we get
\begin{equation}\label{eq:sub1second}
\begin{split}
&\int_{\O_0}(|\nabla u_i|+|\nabla \wt{u}_i|)^{p_i-2}|\nabla (u_i-\tu_i)^+|^2\,dx\\   
&\leq  \int_{\O_0}\big( |\n u_i|^{p_i-2}\n u_i -|\n\wt{u}_i|^{p_i-2}\n \wt{u}_i,
\n (u_i-\tu_i)^+ \big)	\, dx\\ 
& \leq C\int_{\O_0}[f_i(u_1,u_2,\ldots,u_m)-f_i(\wt{u}_1,\wt{u}_2,\ldots,\wt{u}_m)](u_i-\tu_i)^+ \, dx +C \int_{\O_0}[(u_i-\tu_i)^+]^2\,dx,
\end{split}
\end{equation}
where $C=C(p_i,q_i,K_{a_i}, L_{a_i},\|\nabla u_i\|_{L^{\infty}(\Omega)}, \|\nabla \wt{u}_i\|_{L^{\infty}(\Omega)})$ is a positive constant. 
\noindent The first term on the right hand side of \eqref{eq:sub1second} can be arranged  as follows
\begin{equation} \label{lipf}
\begin{split}
&\int_{\O_0}[f_i(u_1,u_2,\ldots,u_m)-f_i(\wt{u}_1,\wt{u}_2,\ldots,\wt{u}_m)] (u_i-\tu_i)^+ \, dx\\
=&\int_{\O_0}[f_i(u_1,u_2,\ldots,u_m)-f_i(\wt{u}_1,u_2,\ldots,u_m)+f_i(\wt{u}_1,u_2,\ldots,u_m)
\\&-f_i(\wt{u}_1,\wt{u}_2,\ldots,\wt{u}_m)] (u_i-\tu_i)^+ \,dx\\
=&\int_{\O_0}[f_i(u_1, u_2,\ldots,u_m)-f_i(\wt{u}_1,u_2,\ldots,u_m) +f_i(\wt{u}_1,u_2,\ldots,u_m)-f_i(\wt{u}_1,\wt{u}_2,\ldots,u_m)\\
&+\ldots+ f_i(\wt{u}_1,\wt{u}_2,\ldots, u_i, \ldots,u_m)-f_i(\wt{u}_1,\wt{u}_2,\ldots, \tilde u_i, \ldots,u_m) +f_i(\wt{u}_1,\wt{u}_2,\ldots, \tilde u_i, \ldots,u_m)  \\
\vdots\\
&\ldots +f_i(\wt{u}_1,\wt{u}_2, \ldots,{u}_m)-f_i(\wt{u}_1,\wt{u}_2,\ldots,\wt{u}_m)] (u_i-\tu_i)^+ \,dx.
\end{split}
\end{equation}
Using the fact that $f_i$ are $\mathcal{C}^1_{loc}$ functions  satisfying  \eqref{coop}, see $({hp^*})$, by \eqref{lipf} we have
\begin{equation} \label{finalf}
\begin{split}
&\int_{\O_0}[f_i(u_1,u_2,...,u_m)-f_i(\wt{u}_1,\wt{u}_2,...,\wt{u}_m)] (u_i-\tu_i)^+ \, dx  
\\
&\leq \int_{\Omega_0}\frac{f_i(u_1,u_2,\ldots,u_m)-f_i(\wt{u}_1,u_2,\ldots,u_m)}{(u_1-\tu_1)^+}(u_1-\tu_1)^+(u_i-\tu_i)^+\, dx
\\
&+ \int_{\Omega_0}\frac{f_i(\tilde u_1,u_2,\ldots,u_m)-f_i(\wt{u}_1,\tilde u_2,\ldots,u_m)}{(u_2-\tu_2)^+}(u_2-\tu_2)^+(u_i-\tu_i)^+\, dx
\\
\vdots
\\
&+ \int_{\Omega_0}\frac{f_i(\tilde u_1,\tilde u_2,\ldots, u_i, \ldots, u_m)-f_i(\wt{u}_1,\tilde u_2,\ldots,\tilde u_i, \ldots, u_m)}{(u_i-\tu_i)}[(u_i-\tu_i)^+]^2\, dx
\\
\vdots
\\
&+\int_{\Omega_0}\frac{f_i(\wt{u}_1,\wt{u}_2, \wt{u}_3\ldots,{u}_m)-f_i(\wt{u}_1,\wt{u}_2,\ldots,\wt{u}_m)}{(u_m-\tilde u_m)^+}](u_m-\tilde u_m)^+ (u_i-\tu_i)^+ \,dx
\\
&\leq L_{f_i} \sum_{j=1}^{m}\int_{\O_0} (u_j-\tu_j)^+(u_i-\tu_i)^+ \, dx,
\end{split}
\end{equation}
where $ L_{f_i}$ is  the Lipschitz constant of $f_i$ that depends on  the
$\displaystyle \max_{1\leq j\leq m }\{\|u_j\|_{L^\infty(\Omega)}\}
$.  Exploiting Young's inequality on the right hand side of \eqref{finalf}, we get
\begin{equation}\label{eq:rayplay}
\int_{\O_0}[f_i(u_1,u_2,...,u_m)-f_i(\wt{u}_1,\wt{u}_2,...,\wt{u}_m)] (u_i-\tu_i)^+ \, dx\leq C \sum_{j=1}^m\int_{\O_0} [(u_j-\tu_j)^+]^2 \, dx,
\end{equation}
where $\displaystyle C=C\big (m, L_{f_i}\big)$ is a positive constant. Finally, from \eqref{eq:sub1second} and \eqref{eq:rayplay} we infer for $i=1,\ldots, m$
\begin{equation}\label{eq:sub222}
\int_{\O_0} (|\n u_i|+|\n \wt{u}_i|)^{p_i-2}|\n (u_i-\tu_i)^+|^2 dx\leq C_i \sum_{j=1}^m \int_{\O_0} [(u_j-\tu_j)^+]^2 \, dx,
\end{equation}
where $\displaystyle C_i=C_i(m,p_i,q_i,K_{a_i},L_{a_i}, L_{f_i},\|\nabla u_i\|_{L^{\infty}(\Omega)}, \|\nabla \wt{u}_i\|_{L^{\infty}(\Omega)})$ is a positive constant.

In the case $p_j\geq 2$, a weighted Poincar\'{e} inequality holds true on  the right hand side of \eqref{eq:sub222}, see Theorem \ref{bvbdvvbidvldjbvlb}. Indeed,  equation \eqref{Sobolev} yields
\begin{equation}\label{eq:pu}
\int_{\O_0} [(u_j-\tu_j)^+]^2dx\leq C_{P, j}(\O_0)\int_{\O_0} (|\n
u_j|+|\n \tilde u_j|)^{p_j-2}|\n (u_j-\tu_j)^+|^2dx,\quad \text{if } p_j \geq 2,
\end{equation}
where  the Poincar\'e constant  $C_{P, j}(\O_0)\to 0$, when the Lebesgue measure $\mathcal{L}(\O_0)\to 0$.  Actually, we used the fact that, since $p_j\geq 2$,
\[|\n
u_j|^{p_j-2}\leq (|\n
u_j|+|\n \tilde u_j|)^{p_j-2}.\]
In the case $p_j <2$, we use the standard Poincar\'{e} inequality on the right hand side of \eqref{eq:sub222}, namely
\begin{equation*}
\int_{\O_0}[(u_j-\tu_j)^+]^2 \, dx\leq 
C_{P, j}(\O_0)\int_{\O_0} |\n (u_j-\tu_j)^+|^2 \, dx,\quad \text{if } p_j < 2,
\end{equation*}
and $C_{P, j}(\O_0)\to 0$ if  $\mathcal{L}(\O_0)\to 0$.
Moreover, in the case $p_j< 2$  since $u_j, \tilde u_j \in C^1(\overline\Omega)$, 
we deduce also  
\begin{eqnarray}\label{eq:martin}
&&\int_{\O_0} |\n (u_j-\tu_j)^+|^2 dx \\\nonumber
&&\leq C(p_j, \|\nabla u_j\|_{L^{\infty}(\Omega)}, \|\nabla \wt{u}_j\|_{L^{\infty}(\Omega)})\int_{\O_0} (|\n u_j|+|\n \wt{u}_j|)^{p_j-2}|\n (u_j-\tu_j)^+|^2 dx. \end{eqnarray}
Using \eqref{eq:martin}, up to redefine the Poincar\'e constant in this case,  we obtain
\begin{equation}\label{eq:pu1}
\int_{\O_0}[(u_j-\tu_j)^+]^2 \, dx\leq 
C_{P, j}(\O_0)\int_{\O_0} (|\n u_j|+|\n \wt{u}_j|)^{p_j-2} |\n (u_j-\tu_j)^+|^2 \, dx,\quad \text{if } p_j < 2,
\end{equation}
and $C_{P, j}(\O_0)\to 0$ if  $\mathcal{L}(\O_0)\to 0$.
Let us set now
\begin{equation}\label{eq:pu11}
C_{P}(\O_0)=\max_{1\leq j\leq m}\{C_{P, j}(\O_0)\}.
\end{equation}
Furthermore, by combining \eqref{eq:sub222} with \eqref{eq:pu}, \eqref{eq:pu1} and \eqref{eq:pu11}, we obtain for $i=1,\ldots,m$
\begin{eqnarray}\label{eq:martin1}
&&\int_{\O_0} (|\n u_i|+|\n \wt{u}_i|)^{p_i-2}|\n (u_i-\tu_i)^+|^2 dx\\\nonumber
&&\leq
C_iC_{P}(\Omega_0)\sum_{j=1}^m\int_{\O_0} (|\n u_j|+|\n \wt{u}_j|)^{p_j-2}|\n (u_j-\tu_j)^+|^2 dx.
\end{eqnarray}
Let us define $\displaystyle \hat C= m\cdot \max_{1\leq i\leq m} \{C_i\}$.
By adding equations  \eqref{eq:martin1} and setting 
\[
I(\O_0)=\sum_{i=1}^m\int_{\O_0} (|\n u_i|+|\n \wt{u}_i|)^{p_i-2}|\n (u_i-\tu_i)^+|^2 dx,
\]
we obtain
\begin{equation}\label{eq:martin2}
I(\O_0)\leq \hat CC_{P}(\Omega_0) I(\O_0).
\end{equation}
Now, we choose $\d>0$ sufficiently small such that the condition $\mathcal L(\O_0)\leq\d$ implies
$$\hat CC_{P}(\Omega_0)<1.$$ Therefore, from \eqref{eq:martin2} we get the desired contradiction, namely  
\[
U_i=(u_i-\wt{u}_i)^+ \equiv 0,
\]
for all $i=1,\ldots,m.$
\end{proof}
\section {Simmetry results for solutions to \eqref{NVsystem}: Proof of Theorem \ref{main1}}\label{sec:main}
\noindent In this section we prove our main result. As we said in the introduction, without loss of generality and for the sake of simplicity, since the problem is invariant with respect to translations, reflections and rotations,  we suppose that $\Omega$ is a  bounded smooth  domain which is convex in the $x_1$-direction and symmetric with respect to $\{x_1=0\}$. 
Let us now recall the main ingredients of the moving plane method.     We set
$$
T_\lambda:=\{x \in \R^N: x_1= \lambda \}.
$$
Given $x\in \R^N$ and $\lambda <0$, we define
\[
x_\lambda=R_\lambda(x):= (2\lambda-x_1, x_2,\ldots,x_N)
\]
and the reflected functions 
\[
u_{i,{\lambda}}(x):=u_i(x_\lambda),\qquad i=1,2,\ldots, m.
\]
We also set 
\begin{equation}\nonumber
\Omega_\lambda:=\{x\in \Omega\,:\, x_1<\lambda\},
\end{equation}
\begin{equation}\label{eq:aainf}
a:=\inf_{x\in \O} x_1,
\end{equation}
\begin{equation}\label{eq:capitalLambda}
\Lambda :=\Big\{a<\lambda<0 \,\, : u_i\leq u_{i,t}\,\, \text{in}\,\, \Omega_t, \,\,\text{for all}\,\, t\in(a,\lambda] \,\, \text{and}\,\,  \text{for all}\,\,i=1,2,\ldots, m \Big \}
\end{equation}
and (if $\Lambda\neq \emptyset$)
\begin{equation}\nonumber
\bar \lambda=\sup\Lambda.
\end{equation}
Finally, for $i=1,\ldots,m$, we define the {\em critical sets} \begin{equation*}
Z_{u_i} := \{x\in \Omega: \n u_i(x)=0\}.
\end{equation*}

\begin{proof}[Proof of Theorem \ref{main1}]
For $a<\lambda<0$ (see \eqref{eq:aainf}) and $\lambda$ sufficiently close to $a$, we  assume that
$\mathcal{L}(\O_\lambda)$ is as small as we need. In particular, we may assume that
Proposition~\ref{pro:confr} works with $\Omega_1=\Omega, \Omega_2= R_\lambda (\Omega),\Omega_0=\Omega_\lambda$ and $\tilde u_i=u_{i,\lambda}$. Therefore, we set
\[
W_{i,\lambda} :=u_i-u_{i,\lambda},\quad  i=1,2,\ldots, m
\]
and we observe that, by construction, we have
\begin{equation*}
W_{i,\lambda}\leq 0\,\,\text{on $\partial\O_\lambda$},\quad  i=1,2,\ldots, m.
\end{equation*}
By Proposition~\ref{pro:confr}, it follows that
\begin{equation*}
W_{i,\lambda} \leq 0\,\,\text{in $\O_\lambda$}, \quad  i=1,2,\ldots, m.
\end{equation*}
Hence, the  set  $\Lambda$ (see \eqref{eq:capitalLambda}) is not empty 
and $\bar \lambda \in (a,0]$.
Note that, by continuity, it follows $u_i\leq u_{i,\bar \lambda}$.
We have to show that, actually $\bar\lambda =0$. Hence, we assume by contradiction that $\bar \lambda < 0$ and we argue as follows. 

First of all, we  point out that  $\mathcal{L}(Z_{u_i})=0$ for all $i$.  Indeed, if we apply  Theorem \ref{local1}, for $u_i$ with  $f(x,u_i)=f_i(u_1,u_2,\ldots,u_i,\ldots, u_m)$,  from~\eqref{drdrdbisssetebissete} the conclusion follows. Hence, let  $A$ be an open set such that for $i=1,\ldots,m$
$$ Z_{u_i}\cap\Omega_{\bar\lambda} \subset A \subset\Omega_{\bar\lambda},$$
with  the Lebesgue measure $\mathcal{L}(A)$ small as we like.
Notice now that, since $f_i$  are locally $\mathcal{C}^1$ functions and $\|u_i\|_{L^{\infty}(\Omega)}\leq C$ for any $i\in \{1,\ldots, m\}$, there exists a
positive constant $\Theta$ such that
\begin{equation}\label{pranzo}
\frac{\partial f_i}{\partial u_i}+\Theta\geq 0\,\,\text{for all $u_1,u_2,\dots,u_m>0$}.
\end{equation}
Furthermore, using  \eqref{coop} we obtain 
\begin{eqnarray}\label{eq:confr}
&&-\Delta_{p_i}  u_i +a_i(u_i)|\nabla u_i|^{q_i}+ \Theta u_i  =f_i (u_1,u_2,\ldots,u_m)+ \Theta u_i 
\\\nonumber
&&\leq  f_i (u_{1,\lambda},u_{2,\lambda},\ldots,u_{m,\lambda})+\Theta u_{i,\lambda}
= -\Delta_{p_i}  u_{i,\lambda}  +a_i(u_{i,\lambda})|\nabla u_{i,\lambda}|^{q_i}+ \Theta u_{i,\lambda}
\end{eqnarray}
for any  $a< \lambda\leq \bar\lambda$.  In light
of~\eqref{eq:confr} we have\begin{equation}\label{eq:sist1}
\begin{cases}
-\Delta_{p_i}  u_i +a_i(u_i)|\nabla u_i|^{q_i}+ \Theta u_i
\leq -\Delta_{p_i}  u_{i,\lambda}  +a_i(u_{i,\lambda})|\nabla u_{i,\lambda}|^{q_i}+ \Theta u_{i,\lambda} & \text{in $\O_\lambda$},\\
u_i \leq u_{i,\lambda} & \text{in $\O_\lambda$}.
\end{cases}
\end{equation}
Then, by~\eqref{eq:sist1} and the strong comparison principle, see statement $(1)$ of  Theorem \ref{SCP}, for any $i=1,2,\ldots,m$ such that  $p_i\geq 2$, we have
\[
u_i < u_{i,\bar\lambda} \qquad\text{or}\qquad u_i \equiv u_{i,\bar\lambda} ,
\]
in  $\Omega_{\bar\lambda}$. 

\

\noindent In the case $1<p_i<2$,    we    prove first the following

\

\begin{center}
{\sc Claim:} The case $u_i\equiv u_{i,\bar\lambda}$ in some  connected component $\mathcal{C}$  of $\Omega_{\bar\lambda}\setminus Z_{u_i}$, such that $\overline{\mathcal C}\subset \Omega$, is not possible.
\end{center}

\

\noindent We proceed by contradiction. Let us assume that such component exists, namely
$$\mathcal C\subset \Omega \quad\text{such that} \quad  \partial \mathcal C\subset Z_{u_i}.$$
For all $\varepsilon>0$, let us define  $G_\varepsilon:\mathbb{R}^+_0\rightarrow\mathbb{R}$ by setting
\begin{equation}\label{eq:G}
G_\varepsilon(t)=\begin{cases} 0 & \text{if $0\leq t\leq \e$}  \\
2t-2\varepsilon& \text{if $\varepsilon\leq t\leq2\varepsilon$}
\\ t & \text{if  $t\geq 2\varepsilon$}.
\end{cases}
\end{equation}
\noindent Let $\chi_{\mathcal A}$ be the characteristic function of a set $\mathcal A$.
We define
\begin{equation}\label{eq:concettinaanalitica}
\Psi_{\varepsilon}\,:=\,e^{-s_i(u_i)}\frac{G_\varepsilon(|\nabla u_i|)}{|\nabla u_i|}\chi_{(\mathcal {C}\cup \mathcal{C}^\lambda)},
\end{equation}
where $\mathcal{C}^\lambda$ is the reflected set of $\mathcal{C}$ with respect to the hyperplane $T_{\bar\lambda}$ and
\begin{equation}\label{eq:aooooooooo}
s_i(t)=\hat C_i\cdot \int_0^t \, a_i^+(t')dt',
\end{equation}
where $a_i^+:=\max\{0, a_i\}$ ($a_i^-:=-\min\{0, a_i\}$) and $\hat C_i$ denotes  some positive constant to be chosen later. 

We point out that supp$\Psi_\varepsilon \subset \mathcal {C}\cup \mathcal{C}^\lambda$, which implies    $\Psi_\varepsilon \in W^{1,p}_0( \mathcal {C}\cup \mathcal{C}^\lambda)$. Indeed by definition of $\mathcal C$ we have that $\nabla u_i=0$ on $\partial(\mathcal {C}\cup \mathcal{C}^\lambda)$. Moreover using the test function $\Psi_{\varepsilon}$ defined in \eqref{eq:concettinaanalitica}, we are able to integrate on the boundary $\partial(\mathcal {C}\cup \mathcal{C}^\lambda)$ which could be not regular. 

Hence, we obtain
\begin{eqnarray}\label{qqqqqqbissete}
&&\int_{\mathcal C\cup \mathcal{C}^\lambda}\, |\nabla u_i|^{p_i-2} (\nabla u_i ,\nabla\Psi_{\varepsilon}) dx+
\int_{\mathcal C\cup \mathcal{C}^\lambda}a_i^+(u_i)|\nabla u_i|^{q_i}\Psi_{\varepsilon} dx\\ \nonumber
&&=\int_{\mathcal C\cup \mathcal{C}^\lambda}a_i^-(u_i)|\nabla u_i|^{q_i}\Psi_{\varepsilon} dx +\int_{\mathcal C\cup \mathcal{C}^\lambda}f_i(u_1,u_2,...,u_m)\Psi_{\varepsilon} dx.
\end{eqnarray}
It is easy to see that for every $x\in [0, M]$ and for every  
$l,q\geq 1$ and $\sigma>0$, there exists a positive constant $C=C(l,q,\sigma, M)$ such that
 \begin{equation}\label{blelellelelle}
x^{q}\leq C \cdot x^l+\sigma,\quad x\in [0,M].
\end{equation}
%$C(\sigma)=\max\{1, \frac{1-\sigma}{\sigma^{\frac{p}{q}}}\}$.
Therefore,
\eqref{qqqqqqbissete} and \eqref{blelellelelle} imply:
\begin{eqnarray}\label{qeqeqeeqeqeq}
&&\int_{\mathcal C\cup \mathcal{C}^\lambda}\, |\nabla u_i|^{p_i-2} (\nabla u_i ,\nabla\Psi_\varepsilon) dx+C_i(\sigma_i, p_i,q_i,\|\nabla u_i\|_{L^{\infty}(\Omega)})\int_{\mathcal C\cup \mathcal{C}^\lambda}a_i^+(u_i)|\nabla u_i|^{p_i}\Psi_\varepsilon dx\\\nonumber&&+\sigma_i\int_{\mathcal C\cup \mathcal{C}^\lambda}a_i^+(u_i)\Psi_\varepsilon dx
\\\nonumber &&\geq\int_{\mathcal C\cup \mathcal{C}^\lambda}a_i^-(u_i)|\nabla u_i|^{q_i}\Psi_\varepsilon dx+\int_{\mathcal C\cup \mathcal{C}^\lambda}f_i(u_1,u_2,...,u_m)\Psi_\varepsilon dx\\\nonumber&&\geq \int_{\mathcal C\cup \mathcal{C}^\lambda}f_i(u_1,u_2,...,u_m)\Psi_\varepsilon dx.
\end{eqnarray}
By $({hp^*})-(ii)$, since $\overline{\mathcal C\cup \mathcal{C}^\lambda}\subset \Omega$ we have that there exists $\gamma_i>0$ such that 
\[f_i(u_1,u_2,...,u_m)\geq \gamma_i.\]
Hence, we can choose $\sigma_i$ in \eqref{blelellelelle}, say $\bar \sigma_i$,
small enough such that
$$
\gamma_i-\bar\sigma_i \,\|a_i^+(u_i)\|_{\infty}=\tilde C_i>0\,,
$$
so that
\begin{eqnarray}\label{qeqeqeeqeqeqAAA}
&&\int_{\mathcal C\cup \mathcal{C}^\lambda}\, |\nabla u_i|^{p_i-2} (\nabla u_i ,\nabla\Psi_\varepsilon ) dx+C_i(\bar\sigma_i, p_i,q_i,\|\nabla u_i\|_{L^{\infty}(\Omega)})\int_{\mathcal C\cup \mathcal{C}^\lambda}a_i^+(u_i)|\nabla u_i|^{p_i}\Psi_\varepsilon  dx\\\nonumber
&&\geq\tilde C_i\int_{\mathcal C\cup \mathcal{C}^\lambda}\Psi_\varepsilon dx.
\end{eqnarray}
Choosing $\hat C_i$ in \eqref{eq:aooooooooo} equal to $C_i(\bar\sigma_i, p_i,q_i,\|\nabla u_i\|_{L^{\infty}(\Omega)})$ in \eqref{qeqeqeeqeqeqAAA} we obtain
\begin{eqnarray}\label{VaScOoOoOoO}\\\nonumber
&&\int_{\mathcal C\cup \mathcal{C}^\lambda}e^{-s_i(u_i)}|\nabla u_i|^{p_i-2}  \left(\nabla u_i,  \nabla \frac{G_\varepsilon(|\nabla u_i|)}{|\nabla u_i|}\right)dx\\\nonumber
&&\geq\tilde C_i\int_{\mathcal C\cup \mathcal{C}^\lambda}e^{-s_i(u_i)}\displaystyle \frac{G_\varepsilon(|\nabla u_i|)}{|\nabla u_i|} dx.
\end{eqnarray}
We set ${\displaystyle h_\varepsilon (t)=\frac{G_\varepsilon(t)}{t}}$, meaning that $h_\varepsilon(t)=0$ for $0\leq t\leq \varepsilon$.
We have:
\begin{eqnarray}\label{eq:smm3}
&&\left| \int_{\mathcal C\cup \mathcal{C}^\lambda}e^{-s_i(u_i)}|\nabla u_i|^{p_i-2}\left(\nabla u_i,  \nabla \frac{G_\varepsilon(|\nabla u_i|)}{|\nabla u_i|}\right)dx\right|\\\nonumber &&\leq \int_{\mathcal C\cup \mathcal{C}^\lambda}|\nabla u_i|^{p_i-1}|h_\varepsilon'(|\nabla u_i|)||\nabla (|\nabla u_i|)|dx\\\nonumber
&&\leq C_i\int_{\mathcal C\cup \mathcal{C}^\lambda}|\nabla u_i|^{p_i-2}\Big(|\nabla u_i|h_\varepsilon'(|\nabla u_i|)\Big)\|D^2 u_i\| dx,
\end{eqnarray}
where $\|D^2 u_i\|$ denotes the Hessian  norm and $C_i$ a positive constant.\\

\

\noindent We let $\varepsilon\rightarrow 0$.
To this aim, let us first show  that
\begin{itemize}
	\item [$(i)$] $|\nabla u_i|^{p_i-2}\|D^2 u_i\| \in L^1(\mathcal C\cup \mathcal{C}^\lambda);$
	
	\
	
	\item [$(ii)$] $|\nabla u_i|h_\varepsilon'(|\nabla u_i|)\rightarrow 0$ a.e. in $\mathcal C\cup \mathcal{C}^\lambda$ as $\varepsilon \rightarrow 0$ and $|\nabla u_i|h_\varepsilon'(|\nabla u_i|)\leq C$ with $C$ not depending on $\varepsilon$.
\end{itemize}
Let us  prove $(i)$. By H\"older's inequality it follows
\begin{eqnarray}\label{eq:smm4}
&&\int_{\mathcal C\cup \mathcal{C}^\lambda}|\nabla u_i|^{p_i-2}\|D^2u_i\| dx\leq \sqrt{\mathcal{L}(\mathcal C\cup \mathcal{C}^\lambda)}\left( \int_{\mathcal C\cup \mathcal{C}^\lambda}|\nabla u_i|^{2(p_i-2)}\|D^2u_i\|^2 dx \right)^{\frac 12}\\\nonumber
&\leq&C_i\left( \int_{\mathcal C\cup \mathcal{C}^\lambda}|\nabla u_i|^{p_i-2-\beta_i}\|D^2u_i\|^2|\nabla u_i|^{p_i-2+\beta_i}dx \right )^{\frac 12}\\\nonumber &\leq& C_i \|\nabla u_i\|^{(p_i-2+\beta_i)/2}_{L^{\infty}(\Omega)} \left(\int_{\mathcal C\cup \mathcal{C}^\lambda}|\nabla u_i|^{p_i-2-\beta_i}\|D^2 u_i\|^2 dx \right)^{\frac 12},
\end{eqnarray}
with  $0\leq\beta_i<1$ and $C_i$ a positive constant. 

\noindent Using \eqref{drdrd} of Theorem \ref{local1}, we infer that
$$\left(\int_{\mathcal C\cup \mathcal{C}^\lambda}|\nabla u_i|^{p_i-2-\beta_i}\|D^2 u_i\|^2 dx \right)^{\frac 12}\leq C.$$ Then, by \eqref{eq:smm4} we obtain
$$\int_{\mathcal C\cup \mathcal{C}^\lambda}|\nabla u_i|^{p_i-2}\|D^2u_i\| dx\leq C.$$

\

\noindent Let us prove $(ii)$. Recalling \eqref{eq:G}, we obtain
$$
h'_\varepsilon(t)=
\begin{cases} 0 & \text{if $0< t \leq \e$}  \\
\frac{2\varepsilon}{t^2}& \text{if $\varepsilon< t<2\varepsilon$}
\\ 0 & \text{if  $t\geq  2\varepsilon$} ,
\end{cases}
$$
and, then, $|\nabla u_i|h_\varepsilon'(|\nabla u_i|)$ tends to $0$
almost everywhere in $\mathcal C\cup \mathcal{C}^\lambda$ as $\varepsilon$ goes to $0$ and
$|\nabla u_i|h_\varepsilon'(|\nabla u_i|)\leq 2$.

Finally, by the Lebesgue's dominate convergence theorem,   passing to the limit for $\varepsilon \rightarrow 0$ in \eqref{VaScOoOoOoO} we obtain
\[
0 \geq \tilde C_i\int_{\mathcal C\cup \mathcal{C}^\lambda}e^{-s_i(u_i)}dx >0.
\]
This gives a contradiction, hence  the {\sc Claim} holds. 

Then, using  also  Hopf's boundary lemma  (see \cite[Theorem 5.5.1]{PSB}) for 
\[-\D_{p_i} u_i+a_i(u_i)|\nabla u_i|^{q_i}= f_i(u_1,u_2,\ldots,u_i,\ldots, u_m)\geq0, \]
$u_i> 0$ in $\Omega$ and $u_i= 0$ on $\partial\Omega$, we deduce that the set $\Omega_{\bar \lambda}\setminus Z_{u_i}$ is connected. Indeed, thanks to Hopf's lemma,  $Z_{u_i}$ lies far from the boundary $\partial \Omega$. {Moreover we also  remark that since $\Omega$ is convex in the $x_1$-direction, we have that  the boundary $\partial \Omega$ is connected.} 
Consequently, for any $i=1,2,\ldots,m$ we get
\begin{equation}
    \label{pippopappo1}
u_i < u_{i,\bar\lambda} 
\end{equation}
in $\Omega_{\bar \lambda}\setminus Z_{u_i}$.

Consider now a compact set $K$ in $\Omega_{\bar \lambda}$ such that
${\mathcal L}(\Omega_{\bar\lambda}\setminus K)$ is sufficiently small so that
Proposition~\ref{pro:confr} can be applied. By what we proved before, for any $i\in \{1,\ldots, m\}$, it holds that $u_i < u_{i,\bar\lambda}$ in $K \setminus A$, which is compact. Then, by (uniform) continuity,
we find $\epsilon>0$ such that, $\bar \lambda+\epsilon < 0$ and for
$\bar \lambda<\lambda<\bar \lambda+\epsilon$ we have that ${\mathcal L}(\Omega_{\lambda}\setminus (K\setminus A))$ is
small enough as before, and $u_{i,\lambda}-u_i >0$ in $K \setminus A$ for any $i$. In particular, $u_{i,\lambda}-u_i >0$ on $\partial(K \setminus A)$.
Consequently, $u_i \leq u_{i,\lambda}$ on
$\partial(\Omega_{\lambda}\setminus(K\setminus A))$. By Proposition~\ref{pro:confr} it follows
$u_i \leq u_{i,\lambda}$  in 
$\Omega_{\lambda}\setminus(K\setminus A)$ and,
consequently in $\Omega_{\lambda}$, which contradicts the assumption $\bar\lambda<0$.
Therefore $\bar \lambda =0$ and the thesis is proved.
Finally,~\eqref{eq:akdkjsjkaslarea} follows by the monotonicity of the
solution that is implicit in the moving plane method.
\

Finally, if $\Omega$ is a ball, repeating this argument along any direction, it follows that $u_i$, $i=1,\ldots,m$, are radially symmetric. The fact that
 $\displaystyle \frac{\partial u_i}{\partial r}(r)<0$ for $r\neq 0$, follows by the Hopf's boundary lemma which works in this case since the level sets are balls
 and, therefore, fulfill the interior sphere condition.
 
 Finally \eqref{eq:derivatapositiva} follows by \eqref{eq:akdkjsjkaslarea} using Theorem \ref{SCP}  (see the statement $(2)$) and the Dirichlet boundary condition of \eqref{NVsystem}. 
\end{proof}

\end{document}